\documentclass[11pt]{amsart}
\usepackage{amscd}                                            
\usepackage[arrow,matrix]{xy}
\usepackage{graphicx, tikz} 
\usepackage{hyperref}
\usepackage{comment}
\usepackage{amsmath}

\usepackage{amsmath, latexsym, amssymb}
\numberwithin{equation}{section}
\theoremstyle{plain}
\newtheorem{lemma}{Lemma}[section]
\newtheorem{proposition}[lemma]{Proposition}
\newtheorem{theorem}[lemma]{Theorem}
\newtheorem{corollary}[lemma]{Corollary}

\theoremstyle{definition}

\newtheorem*{definition*}{Definition}

\usepackage{color}
\usepackage{cancel}

\definecolor{brown}{RGB}{150,100,0}
 
\definecolor{purple}{RGB}{150,0,100}

\definecolor{grey}{RGB}{128,128,128}

\newcommand{\R}{{\mathbb R}}

\newcommand{\Z}{{\mathbb Z}} 
\newcommand{\N}{{\mathbb N}}

\newcommand{\Ff}{{\mathcal F}}

\newcommand{\Xx}{{\mathcal X}}

\newcommand{\eps}{{\varepsilon}}

\def\NABLA#1{{\mathop{\nabla\kern-.5ex\lower1ex\hbox{$#1$}}}}
\def\Nabla#1{\nabla\kern-.5ex{}_#1}

\begin{document}
	\title[Differentiation of measures  on complete  Riemannian  manifolds]{Differentiation of measures  on complete  Riemannian  manifolds}

	\author{J\"urgen Jost}
	\address{Max-Planck-Institute for Mathematics in the Sciences,
	Inselstrasse 22, 04103 Leipzig, Germany}
	\email{jjost@mis.mpg.de}
	
	\author{H\^ong V\^an L\^e }	
	\address{Institute  of Mathematics of the Czech Academey of Sciences,
	Zitna 25, 11567  Praha 1, Czech Republic}
\email{hvle@math.cas.cz}

	\author{Tat Dat  Tran }
	\address{Max-Planck-Institute for Mathematics in the Sciences,
		Inselstrasse 22, 04103 Leipzig, Germany \& 
		Mathematisches Institut, Universit\"at Leipzig, Augustusplatz 10, 04109 Leipzig, Germany
		}
	\email{Tran.Dat@mis.mpg.de \& tran@math.uni-leipzig.de}


\thanks{Research  of HVL was supported  by  GA\v CR-project 18-01953J and	 RVO: 67985840}
\keywords{Besicovitch-Federer covering theorem, differentiation of measure, Radon-Nikodym derivative, complete  Riemannian manifold}
\subjclass[2010]{Primary: 28A15, Secondary: 49Q15, 53C20}

\begin{abstract} In this note  we give a new proof of a version of the  Besicovitch   covering 
	theorem, given  in \cite{EG1992}, \cite{Bogachev2007} and extended in  \cite{Federer1969},  for locally finite Borel  measures  on finite dimensional complete Riemannian manifolds  $(M,g)$. As a  consequence, we   prove a  differentiation   theorem   
	for Borel measures on 
	$(M,g)$, which gives a formula for  the Radon-Nikodym density of two nonnegative  locally  finite Borel measures $\nu_1, \nu_{2}$ on $(M, g)$   such that $\nu_1 \ll  \nu_2$,  extending the  known case   when $(M, g)$ is a  standard  Euclidean  space.
	 
\end{abstract}

\maketitle

\section{Introduction}\label{sec:intr}

The existence  of the Radon-Nikodym derivative is  one of the   most frequently  employed  results
in  probability theory and mathematical statistics.  In  the general case, where  $\nu$, $\mu$ are locally finite
measures   on  a  general measurable space $\Xx$ and $\nu \ll \mu$,  classical proofs  of the existence   of the Radon-Nikodym   derivative $d\nu/d\mu$   are  non constructive, see  e.g.  \cite[p. 429, vol. 1]{Bogachev2007}, \cite[\S 8.7, p. 336]{BBT2008} for historical comments.  For    a class  of  metrizable measurable spaces   $\Xx$,    the theorem of differentiation of measures   with a  constructive proof \footnote{The  proof  of the theorem  of differentiation   of measures on  complete $\sigma$-finite measure  spaces given in \cite[Chapter 8]{BBT2008}  utilizes
	the  existence  of lifting, whose  proof is non constructive.}    yields      not only  the  existence  of the
Radon-Nikodym  derivative, but   also computes the Radon-Nikodym density based on an appropriate metric. As far as we know, that is   the only  way to  get   an explicit  formula  for  the Radon-Nikodym derivative, see  \cite[p. 189]{SG1977},  \cite[p. 56]{Panangaden2009}  for  discussions   on the relation  between
	the  Radon-Nikodym  theorem  and the   theorem of differentiation of measures.

  The main ingredient   of   all known   proofs   of   the  theorem of differentiation of measures is   the    construction (or  the existence)
  of     a  differentiation basis, which  is based   on   a covering theorem.
  All covering theorems are based on the same idea: from an arbitrary cover of a 
  set in a metric space, one tries to select a subcover that is, in a certain sense, as 
  disjointed as possible. According  to \cite[Chapter 1]{Heinonen2001}, there are      three      (types) of covering   theorems:  the basic covering theorem, which is an extension of the classical Vitali theorem for $\R^n$ to     arbitrary metric space,  the  Vitali  covering theorem, which is   an  extension  the classical Vitali theorem   to the case of
  doubling  metric measure spaces,    and  the Besicovitch-Federer  theorem
   that has been    first proved by Besicovitch \cite{Besicovitch1945}  for
   the case  of  $\R^n$  and  then  extended by Federer  for  directionally $(\eps,M)$-limited subsets of    a metric space $X$ \cite[Theorem 2.8.14, p.150]{Federer1969}.  Examples of such     subsets  are compact
   subsets in a Riemannian manifold. The essence of Vitali theorems is that one finds a disjointed subcollection of the sets of a given cover that need not be a cover itself, but that when the radii are all enlarged by a fixed factor, covers everything. The essence of the Besicovitch theorems is to select a subcover so that each point is only covered a controlled number of times. Clearly, such theorems are useful when one has to estimate constants occurring in covering arguments.
   
   The  Besicovitch-Federer covering theorems has been revisited   for the 
   case   of $\R^n$  \cite{Sullivan1994}, \cite{EG1992},   and for any    finite  dimensional  normed vector space, which  results in a variation of the Besicovitch-Federer covering theorem  for arbitrary metric spaces \cite{Loeb1989}, and extended  in  \cite{Itoh2018} for  non  directionally limited  subsets in $\R^n$.

   In our  note   we   give  a     new proof  of the following  version  of  the Besicovitch-Federer theorem.

   \begin{theorem}\label{thm:besi}
   	
    Assume that   $\Ff$  is a collection of  open $4$-proper geodesic balls   in   a complete  Riemannian manifold $(M, g)$   such that  the  set  $A$  of the   centers  of  the balls in $\Ff$ is bounded. Then     one  can find $N \in \N^+$ and  subcollections $\Ff_1, \cdots,  \Ff_{N} \subset \Ff$  each of which  consists  of at most  countably  many disjoint   balls  such that  $A$  is covered   by the  balls  from
   		$\Ff_1 \cup \cdots  \cup \Ff_{N}$.
   	\end{theorem}
Here, {\it $4$-proper} means that the radius of the ball is at most 1/4 of the injectivity radius of its center. 
A particular case  of Theorem \ref{thm:besi}  is the version  of   the   Besicovitch covering  theorem  for  the standard Euclidean space  $\R^n$  \cite{Besicovitch1945}, which has been   formulated as  Theorem 5.8.1
in \cite[p. 361, vol. 1]{Bogachev2007} based on  the   proof  of \cite[Theorem 1.27]{EG1992}.  There  are  three  differences  between   Theorem \ref{thm:besi}  and
Theorem 5.8.1 ibid.:  firstly  we  make  the assumption that $A$ is bounded, secondly,  the  geodesic balls are $4$-proper, and  thirdly,
the balls are open  instead of nondegenerate closed
as in  Theorem 5.8.1 ibid.  Note that  in  the Besicovitch-Federer theorem
\cite[Theorem 2.8.14]{Federer1969}   the similar family $\Ff$  also  consists  of closed balls. (In fact Theorem  \ref{thm:besi}    is  also valid  for closed balls,  but we  need to    track  and change,  if necessary,  several similar    strict or non-strict inequalities  in the proof.)  The main idea   of  our proof is to use  comparison  theorems  in Riemannian geometry to reduce the situation to the Euclidean one.

As a result,  we shall  prove  a     theorem of    differentiation  of measures  for   locally finite  Borel measures   on  complete  Riemannian manifolds \ref{thm:rnr}, which  yields  a  new formula for  the  Radon-Nikodym  derivative, used   in  our paper \cite{JLT2020}.


  
  
Let $\nu_1$ and  $\nu_2$  be  locally  finite   Borel measures  on $(M, g)$   such that $\nu_2 \ll \nu_1$. For
$x\in M$ we denote by  $D_r(x)$  the  open geodesic ball  of  radius $r$ in $M$ with center
in $x$  and   we set
$$\overline{D}_{\nu_1}  \nu_2 (x): = \lim _{r \to 0} \sup \frac{\nu_2 (D_r(x))}{\nu_1 (D_r(x))},$$
$$\underline{D}_{\nu_1} \nu_2 (x): = \lim _{r \to 0} \inf \frac{\nu_2 (D_r(x))}{\nu_1 (D_r(x))},$$
where  we set  $ \overline{D}_{\nu_1} \nu_2  (x)= \underline{D}_{\nu_1} \nu_2 (x) = +\infty$
if  $\nu_1 (D_r(x))=0$ for some $r>0$.

Furthermore  if $\overline{D}_{\nu_1} \nu_2 (x) = \underline{D}_{\nu_1}\nu_2 (x)$ then we denote their common value by 
$$D_{\nu_1}\nu_2 (x) : = \overline{D}_{\nu_1} \nu_2 (x) = \underline{D}_{\nu_1}\nu_2 (x)$$
which  is called {\it the derivative  of $\nu_2$ with respect  to $\nu_1$ at $x$.}

\begin{theorem}\label{thm:rnr}  Let $\nu_1$ and $\nu_2$ be two nonnegative locally finite Borel measures on  a  complete Riemannian  manifold $(M, g)$  such that $\nu_2 \ll \nu_1$.   Then    there  is a measurable   subset $S_0\subset M$  of zero $\nu_1$-measure   such that     the function $D_{\nu_1}\nu_2$ is defined  and finite  on $M \setminus  S_0$. Setting  $\tilde  D_{\nu_1} \nu_2 (x): = 0$ for
	$x \in  S_0$ and  $\tilde D_{\nu_1}\nu_2 (x) : = D_{\nu_1}\nu_2 (x)  $ for
$x \in M \setminus  S_0$,
  the
function $\tilde D_{\nu_1} \nu_2: M \to \R$  is measurable and serves as the Radon-Nikodym density of the measure $\nu_2$ with respect to $\nu_1$.
\end{theorem}

Theorem \ref{thm:rnr}  is   also different  from  Theorem 5.8.8 in \cite[vol.1]{Bogachev2007} in       defining  $\tilde D_{\nu_1}\nu_2$, since  we need  to apply   it   to a  family of Nikodym derivatives  in our  paper \cite{JLT2020}.

\section{Proof of Theorem \ref{thm:besi}}


Assume the  conditions of   Theorem \ref{thm:besi}.
Let $R : = \sup  \{ r:  D_r(a) \in \Ff\}$.  We can find $D_1 = D_{r_1}(a_1)\in \Ff$  with $ r_1  >  3 R/4$.  The balls   $D_j, j >1$,  are chosen  inductively as follows.  Let $A_j: = A \setminus  \cup_{i =1}^{j-1}  D_i$.  If the  set  $A_j$ is empty, then our  construction  is completed and, letting  $J = j -1$  we obtain $J$  balls $D_1, \cdots, D_J$.  If $A_j$ is nonempty, then 
we  choose  $D_j:=  D_{r_j} (a_j) \in \Ff$  such that
$$ a_j \in A_j  \text{ and }  r_j > \frac{3}{4} \sup \{ r: D_r(a)\in \Ff, a \in A_j\}.$$

In the case  of an infinite  sequence  of balls $D_j$  we set $J  = \infty$. 

\

\begin{lemma}\label{lem:claim1} The   balls $D_j$ satisfy the following properties

(a)  if $j > i$  then $r_j \le  4 r_i /3$,

(b) the balls $D_{r_j/3} (a_j)$ are disjoint   and if $J = \infty$  then $r_j \to 0$  as $ j \to \infty$,

(c) $A  \subset  \cup_{j=1}^J D_j$.
\end{lemma}
\begin{proof} Property (a)    follows   from the definition of $r_i$  and the inclusion $a_j \in A_j \subset A_i$.

Property  (b)  is a consequence of  the  following observation.    If $j > i$ then $a_j  \not \in D_i$ and hence by (a)   we have
\begin{equation}\label{eq:est0}
 \rho_g( a_i ,a_j) \ge  r_i >  \frac{r_i}{3} + \frac{r_j}{3}.
 \end{equation}
Since  $A$ is bounded,  $r_j$   goes to $0$  as $j \to \infty$ if  $J = \infty$.

Finally  (c) is obvious  if $J < \infty$.   If $J = \infty$  and $D_r(a) \in \Ff$ then there exists $r_j $ with  $r_j <  3r /4$  by (b). Hence $a \in  \cup_{i =1} ^{j-1}  D_i$ by   our construction of $r_j$.  This completes the  proof of Lemma \ref{lem:claim1}.

\end{proof}

We fix  $k > 1$  and let
\begin{equation}\label{eq:ikm}I_k : = \{ j:  j <k , \,  D_j \cap D_k  \not = \emptyset\}, \:  M_k : = I_k \cap \{ j : r_j \le  3 r_k\}.
\end{equation}

\begin{lemma}\label{lem:claim2}    There is a number  $ c(A)$  independent  of $k$ such that    $\# M_k \le  c(A)$.
\end{lemma}

\begin{proof}  If $j \in M_k$  and $x \in  D_{r_j/3}(a_j)$ then the balls  $D_j$ 
and $D_k$  are open  and have nonempty intersection  and $r_j \le  3 r_k$, hence
 $$\rho_g(x, a_k) \le \rho_g(x, a_j) + \rho_g(a_j ,a_k)  < \frac{r_j}{3}  + r_j + r_k < 5r_k. $$

It follows  that  $D_{r_j/3} ( a_j) \subset D_{5r_k} (a_k)$. Denote by $vol_g$  the Riemannian volume on $(M, g)$.
  By the  disjointness  of $D(a_j,  r_j/3)$  and  the   boundedness  of $A$, taking into  account the Bishop volume comparison theorem \cite[Theorem 15, \S 11.10]{BC1964},   see  also \cite{Le1993} for a  generalization,  there
  exists  a number $c_1(A)$(depending on an upper bound for the Ricci curvature and on the local topology, but the latter will play no role for 4-proper balls) such that

\begin{equation}\label{eq:est2}
vol_g( D_{5r_k}(a_k))  \ge \sum_{j \in M_k} vol_g(D_{r_j/3}(a_j))\ge  c_1(A)\sum_{j \in M_k}  (\frac{r_j}{3})^n.
\end{equation}

Using   property (a) in Claim 1, we   obtain from (\ref{eq:est2})
\begin{equation}\label{eq:est3}
vol_g (D_{5r_k}(a_k)) \ge \sum_{j \in M_k} c_1(A) (\frac{r_k}{4})^n   = \#(M_k)  c_1(A) (\frac{r_k}{4})^n.
\end{equation}
By the  Bishop comparison theorem  there exists  a number $c_2 (A)$ (depending on a lower bound for the Ricci curvature) such that $vol_g(D_{5r_k}) (a_k) \le  c_2 (A)\cdot (5r_k)^n$. In combination  with (\ref{eq:est3})  we obtain
\begin{equation}\label{eq:est4}
\#(M_k)\le \frac{c_2(A)}{c_1(A)} 20 ^n.
\end{equation}
This completes the proof of Lemma \ref{lem:claim2}.
\end{proof}

\begin{lemma}\label{lem:claim3}  There  exists   a number  $d(A)$ independent of $k$ such that   $\#(I_k \setminus M_k) \le  d(A)$.
\end{lemma}

\begin{proof} 
Let us  consider   two  distinct  elements  $i, j \in I_k \setminus M_k$. By (\ref{eq:ikm})  we have
\begin{equation}\label{eq:est5}
 1 < i, j   < k, \:   D_i \cap D_k  \not = \emptyset, \: D_j \cap D_k \not= \emptyset, \:  r_i > 3 r_k, \:   r_j > 3 r_k.
 \end{equation}
For notational simplicity   we shall   redenote $\rho_g(a_k, a_i) $ by $|a_i|$.  Then  (\ref{eq:est5})  implies
\begin{equation}\label{eq:est6}
|a_i|  <   r_i + r_k \text{ and } |a_j|  <     r_j  + r_k. 
\end{equation}
Let  $\theta_{def}(a_i, a_j) \in [0,\pi]$ be the  deformed angle  between  the two geodesic  rays $(a_k,a_i)$ and $(a_k, a_j)$,  connecting $a_k$ with   $a_i$   and $a_j$ respectively, which  is defined  as  follows
$$ \theta_{def}(a_i, a_j): = \arccos   \frac{|a_i|^2  + |a_j|^2  - \rho_g(a_i , a_j) ^2}{2|a_i| |a_j|}.$$

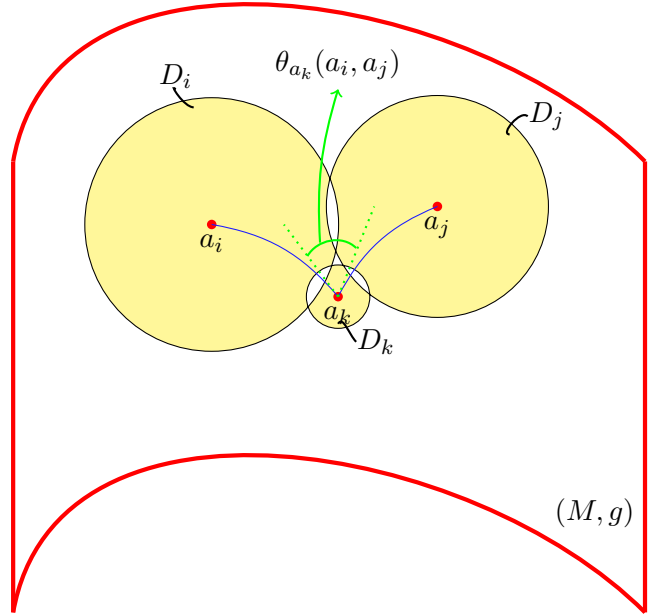
\begin{figure}[htb!]
\begin{center}
\begin{tikzpicture}[scale=1.2]
\node[left] at (3,1.1) {$(M,g)$} ; 
\draw[color=red,ultra thick] (-4, 5)  to [out=80] (3,5);
\draw[color=red,ultra thick] (-4, 0)  to [out=80] (3,0);
\draw[color=red,ultra thick] (-4, 0) -- (-4,5);
\draw[color=red,ultra thick] (3, 0) -- (3,5);


\coordinate (Ai) at (-1.8,4.3);
\coordinate (Aj) at (.7,4.5);
\coordinate (Ak) at (-0.4,3.5);

 \filldraw[fill=yellow!50, even odd rule]  (Ak) circle (10pt) (Ai) circle (40pt) (Aj) circle (35pt);
  \fill[red] (Ak) circle (1.5pt);
  \node[below] at (Ak) {$a_k$} ; 
  
  \fill[red] (Ai) circle (1.5pt);
  \node[below] at (Ai) {$a_i$} ;  

  \fill[red] (Aj) circle (1.5pt);
  \node[below] at (Aj) {$a_j$} ;              
    
  \draw[thick] (-2.0, 5.6)  to [out=180] (-2.2,5.8);
  \node[above] at (-2.2,5.7) {$D_i$} ; 
                
 \draw[thick] (1.5,5.3)  to [out=180] (1.7,5.5);
  \node[right] at (1.6,5.5) {$D_j$} ; 
  
   \draw[thick] (-.4,3.2)  to [out=60] (-0.2,3);
  \node[right] at (-.3,3) {$D_k$} ; 
  
  \draw[thick, green] (-.74,3.95)  to [out=60] (-.2,4.04);
  \draw[thick, green,->] (-0.6,4.1) to [bend left=10] (-0.4,5.8);
  \node[above] at (-0.4,5.8) {$\theta_{a_k}(a_i, a_j)$};


   \draw[thin, blue!90] (Ai)  to [bend left = 20] (Ak);
   \draw[thin, blue!90]  (Ak)  to [bend left=20] (Aj);
   \draw[thick, dotted, green] (Ak) to (-1,4.3);
  \draw[thick, dotted, green] (Ak) to (0,4.5);
 
 \end{tikzpicture}
 \end{center}
 \caption{The $\theta_{def}(a_i,a_j)$ vs $\theta_{a_k}(a_i, a_j)$ (defined below).}
  \end{figure}

  We shall prove the   estimate
\begin{equation}\label{eq:est7}
\theta_{def}(a_i, a_j)  \ge \theta_0 : = \arccos 61/64 >0.
\end{equation}
By  the construction, see also (\ref{eq:est0}), we have $a_k \not \in  D_i \cup D_j$  and  $r_i \le |a_i|$, $ r_j \le  |a_j|$.  W.l.o.g. we assume that   $|a_i|\le |a_j|$.  By (\ref{eq:ikm})   and (\ref{eq:est6})   we  obtain
\begin{equation}\label{eq:est8}
3r_k < r_i \le  |a_i| <  r_i + r_k, \:   3r_k < r_j \le  |a_j| <  r_j + r_k, \:  |a_i| \le |a_j|.
\end{equation}
\

We need   two more claims  for the proof   of (\ref{eq:est7}).

\

{\it Claim 1.} If  $\cos  \theta_{def}(a_i, a_j) > 5/6$ then   $ a_i \in D_j$.

{\it Proof of Claim 1.}  It suffices to show  that   if $a_i \not \in D_j$ then
$\cos \theta_{def}(a_i, a_j) \le 5/6$. Assume  that $a_i \not \in D_j$.   We shall consider   two possibilities, first  assume that $\rho_g(a_i , a_j) \ge  | a_j|$. Then
our assertion follows from the following estimates
\begin{equation}\label{eq:est9}
\cos \theta_{def} (a_i, a_j)= \frac{|a_i|^2  + |a_j|^2  - \rho_g(a_i , a_j)^2}{2|a_i| |a_j|}\le \frac{|a_i|}{2 | a_j|} \le \frac{1}{2} < \frac{5}{6}.
\end{equation}
Now assume  that  $\rho_g(a_i, a_j) \le    |a_j|$. Then  
\begin{eqnarray}\nonumber %
\cos \theta_{def}(a_i, a_j)  = \frac{|a_i|^2 +|a_j|^2  - \rho_g(a_i, a_j)^2}{2 |a_i| |a_j|} \le 
\nonumber\\
\frac{|a_i|}{2|a_j|} + \frac{(|a_j|  - \rho_g(a_i, a_j))( |a_j| + \rho_g(a_i, a_j))}{2 |a_i||a_j|}\nonumber \\
 \le \frac{1}{2} + \frac{|a_j| - \rho_g(a_i, a_j)}{ |a_i|}\nonumber\\
 \le    \frac{1}{2} + \frac{r_j + r_k - r_j} {r_i} \le \frac{5}{6}	\label{eq:est10}
\end{eqnarray}
 where in the   second inequality we  use  the assumption  $ |a_j| +  \rho_g(a_i, a_j) \le 2 |a_j|$,
   in the
 third inequality  we use   $|a_j| \le  r_j + r_k$   and   taking into account  
 $a_i \not \in D_j$ we have  $r_j  \le \rho_g(a_i, a_j) $,   we  also use   $r_i  \le |a_i|$  from (\ref{eq:est8}),
 and  in the last inequality  we use $3r_k < r_i$ from  (\ref{eq:est5}).
 This completes  the proof of Claim 1.
\

{\it  Claim 2.}   If $a_i \in D_j$ then
\begin{equation}\label{eq:est11}
0 \le  \rho_g(a_i, a_j) + |a_i| - |a_j|  \le \frac{8}{3}(1- \cos  \theta_{def}(a_i, a_j))|a_j|.
\end{equation}

{\it Proof of Claim 2.}    We   utilize the  proof  of \cite[(5.8.3), p. 363, vol. 2]{Bogachev2007} . Since $a_i \in D_j$ we  have $i < j$.  Hence $a_j \not \in  D_i$  and therefore  $\rho_g(a_i, a_j) \ge  r_i$.  Keeping  our convention  that $|a_i| \le |a_j|$ we have
$$ 0 \le \frac{\rho_g(a_i, a_j) + |a_i| - |a_j|}{|a_j|} \le \frac{\rho_g(a_i, a_j) + |a_i|  - |a_j|}{|a_j|} \frac{\rho_g(a_i, a_j)-|a_i| + |a_j|}{\rho_g(a_i, a_j)} $$
$$ = \frac{\rho_g(a_i, a_j)^2 - (|a_j|  -|a_i|)^2}{|a_j| \rho_g(a_i, a_j)}  = \frac{2|a_i|(1-\cos \theta_{def}(a_i, a_j))}{\rho_g(a_i, a_j)}$$
$$ \le \frac{2 (r_i + r_k)(1-\cos \theta_{def}(a_i, a_j))}{r_i} \le \frac{8}{3}(1-\cos \theta_{def}(a_i, a_j)).$$
Here in the inequality before  the last we use  the above inequality  $r_i <  \rho_g(a_i, a_j)$  and  $|a_i| < r_i + r_k$ from (\ref{eq:est8}). This completes the  proof of Claim 2.

\

{\it Continuation of the proof of (\ref{eq:est7}).}  If $\cos  \theta_{def} (a_i, a_j)\le 5/6$,  then\\  $\cos \theta_{def}(a_i, a_j) < 61/64$. If $\cos \theta _{def} (a_i, a_j)> 5/6$  then $a_i \in D_j$  by  Lemma \ref{lem:claim1}.
Then $i < j$ and hence $a_j \not \in D_i$. It follows that  $r_i\le \rho_g(a_i, a_j) < r_j$. Recall  by Lemma \ref{lem:claim1} (a) $r_j  \le 4r_i/3$.  Taking into account $ r_j  > 3r_k$ from (\ref{eq:ikm})   we obtain
$$\rho_g(a_i, a_j) +|a_i| - |a_j| \stackrel{(\ref{eq:est8})}{ \ge } r_i + r_i - r_j -r_k  \ge \frac{r_j}{2} - r_k  \ge   \frac{1}{8} (r_j + r_k)  \ge \frac{1}{8} |a_j| $$
which   in combination with (\ref{eq:est11}) yields  
$|a_j|/8  <  8 (1 - \cos \theta_{def}(a_i, a_j)) |a_j|/3$. Hence  $\cos \theta_{def}(a_i, a_j) \le 61/64$. This completes  the proof of   estimate  (\ref{eq:est7}).

\

 In the  next step we  shall   prove  the existence   of  a  lower bound    for the  angle   $\theta_{a_k}(a_i, a_j)$ between
  the two  geodesic  rays $(a_k, a_i)$ and $(a_k, a_j)$, namely  $\theta_{a_k}(a_i, a_j)$ is the angle between   two vectors $\vec{a_i}$ and $\vec{a_j}$  on the  tangent space $T_{a_k} M^n$  provided  with the    restriction of  the  metric $g$  to $T_{a_k}  M^n$, where $\vec{a_i}$ (resp. $\vec{a_j}$)  is
  the  tangent  vector    at $a_k$ of the  geodesic  $(a_k, a_i)$ (resp.   of the geodesic  $(a_k, a_j)$.)
  
  \
  
  {\it Claim 3.}  There exists  a  positive   number  $\alpha(A)$  independent  of  $k, i, j$ such that  $\theta_{a_k}(a_i, a_j) \ge  \alpha (A)$.
  
  {\it Proof  of  Claim 3.}   Since  $A$  is bounded, by the Bishop-Crittenden comparison theorem \cite[Theorem 15, \S11.10]{BC1964}  that estimates  the  differential of the exponential map  via the  sectional  curvature of the  Riemannian manifold  there  exists  a constant  $b(A)$   independent of $a_i, a_j, a_k$  and  sectional curvature  bounds    for  $A \subset  M$  such that
  $\theta_{a_k}(a_i, a_j) >  b(A)\cdot \theta_{def} (a_i, a_j) $.  Combining this with  (\ref{eq:est7})   implies  Claim 3.

{\it Continuation of the proof of  Lemma \ref{lem:claim3}}. Denote by ${\rm inj\, rad}_M(x)$ the injectivity radius  of $M$  at $x$. Let  $r_A : =\inf_{x\in A} {\rm inj\, rad}_M (x)$. 
	Since  $A$  is bounded,    $r_A >0$.

$\bullet$  Let $\delta(A)$ be the largest     positive  number such 
 that: 
 
 (i)  $\delta (A) \le r_A /8$,  
 
 (ii)  For any   $x\not = y\not = z\not =x \in A$    satisfying the following  relations
 $$ \rho_g(x, y) \le \frac{r_A}{4} \text{ and } \rho_g (y, z) \le   \rho_g(x, y) \cdot \delta (A)$$
 we have  $\theta _{x}(y, z)\le \alpha (A)$.
 
 The  existence  of $\delta(A)$  follows  from   the boundedness  of $A$  and  the Bishop-Critenden comparison theorem.
 
 $\bullet$  Let  $d(A)$ be   the smallest  natural   number    such that   for any  $x \in A$   and  any  $r \in (0, r_A/4)$  we can cover  the  geodesic sphere  $S(x, r)$   of radius  $r$  centered  at $x$  by at most  $d(A)$ balls   of   radius $r \cdot  \delta (A)$. The  existence of  $ d(A)$   follows from the boundedness  of $A$  and   Bishop's comparison theorem (lower Ricci bound).

  Claim 3   implies  that
  $\#(I_k \setminus M_k)  \le d(A)$.  This completes  the proof of
 Lemma \ref{lem:claim3}.
\end{proof}

 {\it Completion of the  proof  of  Theorem \ref{thm:besi}.}   Lemmas \ref{lem:claim2} and \ref{lem:claim3} imply that   $\# (I_k)  \le  c(A) + d(A)$.

  Now we  make  a choice  of  $\Ff_i$   in the same  way as  in  the proof of  Theorem 5.8.1 ibid.    
   Set  $L(A): = c(A) + d(A)$.
  We define  a mapping 
  $$\sigma: \{ 1, 2, \cdots, \} \to \{ 1, \cdots, L(A)\}$$
  as follows: $\sigma (i) = i$ if $ 1 \le i \le  L(A)$.  If $k \ge L(A)$, we define  $\sigma(k+1)$ as follows.  Since
  $$\#(I_{k+1}) =\# \{j|\,  1\le j \le k, D_j\cap D_{k+1} \not =\emptyset \} < L(A)$$
   there exists a smallest number  $l \in \{ 1, \cdots,  L(A)\}$ with
   $D_{k+1} \cap D_j = \emptyset$  for all $j \in  \{ 1, \cdots, k\}$  such that $\sigma(j) = l$.  Then we set $\sigma(k+1) : =  l$. Finally, let
   $$\Ff_j : = \{ D_i: \sigma(i) = j\} , j \le  L(A).$$
   By definition of $\sigma$, every collection  $\Ff_j$  consists of  disjoint balls. Since every   ball  $D_i$ belongs  to some collection $\Ff_j$, we have
   $$A \subset  \bigcup_{j=1}^J D_j = \bigcup_{j=1}^{L(A)}\bigcup_{D\in D_j} D.$$
This completes  the proof  of Theorem \ref{thm:besi}
 

\section{Proof  of Theorem \ref{thm:rnr}}\label{sec:mthm}

The proof   of  Theorem \ref{thm:rnr}  uses  the  argument  in the proof  of \cite[Theorem 5.8.8, p. 368, vol. 1]{Bogachev2007},  based  on \cite{EG1992},  with a  modification  to  deal with  a    general complete
Riemannian metric $g$. Furthermore,  assuming the conditions   in Theorem \ref{thm:rnr},   we     modify  $D_{\nu_1}\nu_2$  a bit to get  a function $\tilde D_{\nu_1}\nu_2$  defined 
on $M$. This  is necessary  for dealing with   a family  of    Radon-Nikodym derivatives, considered in \cite{JLT2020}.

First  we shall show  that  $D_{\nu_1}\nu_2 (x)$  exists  and  is finite  for $\nu_1$-a.e..   Let $S : = \{x:  \overline D_{\nu_1}\nu_2 (x) = +\infty  \} $.  We denote  by $\mu^*$  the outer measure   defined  by a locally finite  Borel measure $\mu$ on $M$. To show
$\nu_1 (S) = 0$   we  need  the following

\begin{proposition}\label{prop:est1}  Let  $0< c < \infty$  and $A$    a subset  of $M$.
	
	(i)  If $A  \subset \{ x:  \underline{D}_{\nu_1}\nu_2(x) \le c \}$  then 
	$\nu ^*_2 (A) \le  c\, \nu_1^*(A)$.
	
	(ii)   If $A  \subset \{ x:  \overline{D}_{\nu_1}\nu_2(x) \ge c \}$  then   $\nu ^*_2 (A) \ge  c\, \nu_1^*(A)$.
\end{proposition}

Proposition \ref{prop:est1}  is an   extension  of \cite[Lemma 5.8.7, vol. 1]{Bogachev2007}  and   will be  proved in  a similar way based  on  Theorem  \ref{thm:besi} and  Lemma \ref{lem:outer} below.   We  shall say that  an open geodesic  ball  $D_{r}(x)\subset (M, g)$ is {\it $k$-proper}, if
$kr$ is  at most  the injectivity  radius  of $(M, g)$ at $x$.

\begin{lemma}\label{lem:outer}  Let $\mu$ be a locally finite Borel     
	measure  on  a  complete  manifold $(M, g)$. Suppose that  $\Ff$ is a collection  of  open $4$-proper geodesic balls in $(M, g)$  the  set  of
	centers of which is denoted by $A$,  and for every $a \in A$  and every $\eps >0$,  $\Ff$ contains  an open $4$-proper geodesic ball $D_r(a)$  with $ r< \eps$.  If $A$ is bounded then for every  nonempty  closed  set $U \subset M$ one can find  an at most countable  collection  of disjoint balls $D_j \in \Ff$  such that
	$$ \bigcup _{j=1}^\infty  D_j  \subset U  \text{  and  }  \mu^* ((A\cap U) \setminus \bigcup_{j=1}^\infty D_j ) = 0.$$
\end{lemma}
\begin{proof}[Proof of Lemma \ref{lem:outer}] 
We prove   Lemma  \ref{lem:outer}  using  Theorem \ref{thm:besi}  and  the  Bishop comparison theorem as well as  arguments  in the  proof  of \cite[Corollary 5.8.2, p. 363]{Bogachev2007}.  Let $A$ , $\Ff$  and $U$  be  as in  Lemma \ref{lem:outer}.     By   Theorem \ref{thm:besi}   there exist  subcollections $\Ff_j$ such   that  $\Ff_j$  consists   of at most    countably many disjoint  balls 
	 and	
$$	 A \subset   \bigcup_{j =1}^{L(A)} \bigcup _{D \in \Ff_j} D.$$
Set
$$\Ff^1: =  \{  D \in  \Ff|\, D \subset  U\}.$$
Now we  shall  apply Theorem \ref{thm:besi}  to $A\cap U$  and $\Ff^1$.  Then  we have
$$(A \cap  U) \subset  \bigcup_{j =1} ^{ L(A\cap U)} \bigcup _{  D \in \Ff ^1_j} D.$$

{\it  Claim 4}. We  can choose  $L(A\cap  U) \le L(A)$.
\begin{proof}[Proof  of  Claim 4]
 Since $ A \cap U  \subset A$, 
     we can  choose the constant $c_1(A\cap U)$  (resp.  $c_2 (A\cap U)$,  $d (A \cap U)$, $\alpha (A \cap U) $,  $b (A \cap U)$)     equal  to  $c_1 (A) $ (resp.    $c_2 (A)$,   $ d(A) $, $\alpha(A)$,    $b(A)$)   such that    the  statements   in  the proof  of  Theorem \ref{thm:besi} holds  for  $A \cap U$  with        these (modified)  constants.    Since  $A\cap U \subset  A$,     we  have
 $r_{A\cap U}  \ge r_A$, hence  we  can  also  choose  $\delta (A\cap U): = \delta (A)$, and therefore  $d(A\cap  U):= d (A) $  such that    the  statements   in  the proof  of  Theorem \ref{thm:besi} holds  for  $A \cap U$  with        these (modified)  constants. This proves  Claim 4.  
\end{proof} 

It follows  that
$$\mu^*(A \cap U) \le \sum_{j=1}^{L(A)}\mu^*( (A\cap U) \cap (\bigcup _{D \in \Ff^1 _j} D).$$
Hence  there  exists $j \in \{ 1, \cdots, L(A)\}$ such that
$$\mu^* ( (A \cap U) \cap ( \bigcup_{D \in \Ff^1 _j}D)) \ge {1\over L(A)}  \mu^*  (A\cap U).$$
Therefore  there   exists  a finite    collection  $D_1, \cdots, D_{k_1} \in \Ff^1_j$    such  that
$$\mu^* ( (A \cap U)  \cap ( \bigcup_{i=1}^{k_1} D_i))  \ge{1\over  2L(A)}\mu^* (A\cap U).$$
It follows
\begin{equation}\label{eq:eps1}
\mu^* (  ( A \cap U) \setminus   \bigcup_{j =1} ^{k_1}D_j)\le  ( 1- {1\over 2L(A)})\mu^* (A\cap U).
\end{equation}

Now we set  $U_2 : =  U \setminus   \cup _{j= 1}^{k_1}D_j$. Set
$$ \Ff^2 : = \{ D\in \Ff  | \, D  \subset U_2\}.$$
 The  set
$U_2$ is closed, since  $D_j$ are open.  Hence   there   exists a finite collection  of disjoint  balls $D_{k_1+1}, \cdots, D_{k_2}$  from $\Ff^2$    and  by  (\ref{eq:eps1})  we have
$$\mu^* ( (A \cap U) \setminus \bigcup _{ j =1} ^{k_2} D_j ) = \mu^* ( ( A \cap U_2 )\setminus \bigcup_{j =k_1 +1}^{k_2} D_j) $$
$$\le   (1- {1\over 2L (A)})\mu^* (A \cap U_2 ) \le  (1- {1\over 2L(A)})^2 \mu^*(A \cap U).$$
Repeating   this process  we get    for  all $p \in  \N^+$
$$\mu^*( (A \cap U )\setminus  \bigcup_{j =1} ^{k_p} D_j ) \le  (1 - {1\over 2L(A)})^p  \mu^*(A\cap U).$$
Since  $\mu^* (A)< \infty $    this    proves Lemma \ref{lem:outer}.
	\end{proof}

\begin{proof}[Proof  of Proposition \ref{prop:est1}]
 By the property of outer measures  it suffices  to prove  Proposition  \ref{prop:est1}  for  bounded  sets $A$.    We shall derive 
Proposition  \ref{prop:est1} from  Lemma \ref{lem:outer}   as  in the proof of \cite[Lemma  5.8.7, p. 368,  vol 1]{Bogachev2007}.   
Assume that  $A  \subset \{ x : \, \underline D_{\nu_1} \nu_2  \le c\}$.  Let $\eps >0$ and
$U$ be   a closed set containing  $A$. Denote by $\Ff$ the class  of  all open 4-proper  geodesic  balls $D_r(a)\subset U$  with $ r >0$,   $a\in A$  and $\nu_2 (D_r (a))\le ( c+ \eps) \nu_1 (D_r(a))$.   By the definition   of $\underline D_{\nu_1}\nu_2 $ we   have $ \inf \{ r:  D_r (a)\in \Ff\} = 0$   for all $a \in A$.  By Lemma \ref{lem:outer}  there exists an  at most   countable family  of disjoint balls
$D_j \in \Ff$  with  $\nu_2^* (A \setminus \cup_{j=1}^{\infty}D_j) = 0$  and
$\cup_{j =1}^\infty D_j \subset U$.  Hence
$$\nu_2 ^* (A)\le \sum_{j =1}^\infty \nu_2  (D_j)\le (c +\eps)\sum_{j=1}^\infty \nu_1 (D_j)\le  (c+\eps)\nu_1 (U).$$
Since $U \supset A$ is arbitrary,  we  obtain the desired  estimate.  The second assertion (ii)  is proven similarly, one has  only to take   for $\Ff$  the class of balls  that satisfy  $\nu_2(D_r(a)) \ge (c-\eps)\nu_1 (D_r(a))$. This completes  the proof of Proposition \ref{prop:est1}.
	\end{proof} 

{\it Completion  of the proof of Theorem \ref{thm:rnr}}.   Proposition  \ref{prop:est1} implies  that $\nu_1^* (S) = 0$. 

Next let $  0 < a < b $  and   set
$$ S(a,b): \{  x : \underline{D}_{\nu_1} \nu_2(x) < a  < b  < \overline{D}_{\nu_1}\nu_2 (x) < +\infty\} .$$

Proposition  \ref{prop:est1}  implies  that 
$$b \, \nu_1^* ( S(a, b)) \le \nu_2^* (S(a,b))\le a\, \nu_1 ^* (S(a,b)).$$
Hence $\nu_1  ^* (S(a,b)) = 0$ because $a < b$.   The union  $S_1$
of $S(a,b)$ over all positive rational numbers  $a, b$   also has  zero $\nu_1^*$-measure.
Hence  there  exists  a measurable      subset  $S_0 \subset M$  of zero  $\nu_1$-measure such that
$S \cup S_1 \subset  S_0$. This proves   the 
first assertion  of Theorem \ref{thm:rnr}.

\

Now let us  show that $\tilde D_{\nu_1} \nu_2 (x)$ is    measurable. 
Clearly, it suffices  to show   that $D_{\nu_1} \nu_2 : M^n \setminus  S_0 \to  \R$  is  measurable.

\begin{lemma}\label{lem:meas}
For   each $  r>0$  the function   $f_r(x) : = \nu_1(D_r(x)): M^n \to \R$ is  lower-semi  continuous  and hence measurable.
\end{lemma}
\begin{proof}  Since $\lim_{k \to \infty} \nu_1 (D_{r-1/k} (x )) = \nu_1 (D_r(x) )$, taking into account that  $D_{r-1/k}(x)\subset D_r(y)$ if $| x- y|  < 1/k$, we obtain
$$\lim  _{ y \to x} \inf \nu_1  (D_r(y))  \ge  \nu_1 (D_r (x))$$
which we needed  to prove.
\end{proof}

Since $S_0$ is    measurable, we obtain  immediately   from   Lemma \ref{lem:meas}  the following

\begin{corollary}\label{cor:meas}
For each  $r>0$ the  restriction  $f_r|_{M \setminus  S_0}	$ is a measurable  function.
\end{corollary}

  In the same way,    the  restriction of  function $f'_r (x): = \nu_2(D_r(x))$
  to $M \setminus  S_0$ is   measurable.
For  $k \in  \N^+$   and $x   \in M \setminus S_0$  we set
$$ \tau_k (x) : = \frac{\nu_2 (D_{1/k} (x))}{\nu_1(D_{1/k} (x))}.$$
It follows that  the  function $\tau_k: M\setminus S_0 \to \R$  is  measurable.
Hence the function $D_{\nu_1} \nu_2 (x):M \setminus S_0 \to \R$ is measurable, which we had  to  prove.

Finally  we  prove   that  $\tilde D_{\nu_1}\nu_2$   serves   as the Radon-Nikodym  derivative  of $\nu_2$ w.r.t.  $\nu_1$.
Equivalently we need to show  that
for   for any $A \in \Sigma_{M}$ we have
\begin{equation}\label{eq:rn1}
\nu_2 (A)= \int _A D_{\nu_1}\nu_2 d(\nu_1).
\end{equation}
Here  we  use  the argument in \cite[p. 368-369, vol.1]{Bogachev2007}.  Let $t >1$  and set   for  $ m \in \Z$
$$A_m : = A \cap \{ x\in (M\setminus S_0)|\,  t^m < D_{\nu_1}\nu_2 (x) <  t^{m+1}\}  .$$
The union $\cup_{m = -\infty}^\infty A_m$ covers $A$  up to $\nu_2$-measure zero set, since
$\nu_2$-a.e.   we have  $ D_{\nu_1}\nu_2  >0$.
Hence we have
$$\nu_2(A)= \sum_{m=-\infty}^\infty \nu_2(A_m)\le \sum_{m= -\infty}^\infty  t^{m+1} \nu_1 (A_m)$$
$$ \le    t  \sum_{m=-\infty}^\infty\int_{A_m} \tilde D_{\nu_1}\nu_2 d\nu_1 = t \int_A  D_{\nu_1}\nu_2d\nu_1 .$$
This is true   for any $ t>1$. Hence
\begin{equation}\label{eq:est1}
\nu_2  (A) \le \int_A  D_{\nu_1}\nu_2 d\nu_1.
\end{equation}
Using  $\nu_2(A_m) \ge t^m  \nu_1  (A_m)$ we obtain
\begin{equation}\label{eq:estf}
\nu_2 (A)\ge \int_A D_{\nu_1}\nu_2 d\nu_1.
\end{equation}
Clearly  (\ref{eq:rn1})  follows from (\ref{eq:est1})  and (\ref{eq:estf}).

   This completes
the proof  of  Theorem \ref{thm:rnr}.

\section*{Acknowledgement} The authors  would like to thank Juan Pablo Vigneaux  for helpful comments  on  an  earlier version  of this   note.

\end{document}